\documentclass[journal]{IEEEtran}

\newcommand{\beq}{\begin{equation}}
\newcommand{\eeq}{\end{equation}}
\newcommand{\bsq}{\begin{subequations}}
\newcommand{\esq}{\end{subequations}}
\newcommand{\bq}{\begin{eqnarray}}
\newcommand{\eq}{\end{eqnarray}}
\newcommand{\bqn}{\begin{eqnarray*}}
\newcommand{\eqn}{\end{eqnarray*}}
\usepackage{bbm}
\usepackage[pdftex]{graphicx}
\graphicspath{{Figs/}}
\usepackage{enumerate}
\usepackage{enumitem} 
\usepackage[hidelinks]{hyperref}
\usepackage{cite}

\usepackage[utf8]{inputenc}
\usepackage[hidelinks]{hyperref}
\usepackage{bm}
\usepackage{booktabs}
\DeclareMathAlphabet{\mathcal}{OMS}{cmsy}{m}{n}
\usepackage{helvet}         
\usepackage{courier}        
\usepackage{type1cm}        
\usepackage{makeidx}         
\usepackage{graphicx}        
\usepackage{url} 
\usepackage{multicol}        
\usepackage[bottom]{footmisc}
\usepackage{ifthen}

\usepackage[usenames,dvipsnames]{color}
\usepackage{caption}

\makeindex             

\usepackage{graphics} 
\usepackage{graphicx} 
\usepackage{epsfig} 
\usepackage{epstopdf}
\usepackage{times} 
\usepackage[cmex10]{amsmath}

\usepackage{mathtools}  
\usepackage{amssymb}  

\usepackage{amsthm}
\allowdisplaybreaks[4]

\usepackage{amsfonts}
\usepackage{cite}
\usepackage{verbatim}
\usepackage{comment}
\usepackage{algorithm}   
\usepackage{algorithmic}
\usepackage{multirow}
\usepackage{cite}
\usepackage{stfloats}
\usepackage{supertabular}
\usepackage{longtable}
\usepackage{caption}
\usepackage{subcaption}

\DeclareGraphicsExtensions{.pdf,.png,.jpg,.eps,.ps}

\usepackage{arydshln}
\usepackage{multicol}
\usepackage{colortbl}
\theoremstyle{definition}
\newtheorem{theorem}{Theorem}

\newtheorem{proposition}{Proposition}

\theoremstyle{definition}

\newtheorem{remark}{Remark}

\ifCLASSINFOpdf
\else
\fi

\usepackage{cite}
\usepackage{comment}
\usepackage{ifthen}
\newboolean{showcomments}
\setboolean{showcomments}{true}

\usepackage{amsmath}
\allowdisplaybreaks[4]

\begin{document}

\title{Contextual Robust Optimization for AI Data Center Scheduling with Statistical Guarantees}

\author{Yijie Yang, Xi Weng, and Yue Chen, \IEEEmembership{Senior Member, IEEE}
\thanks{Y. Yang and Y. Chen are with the Department of Mechanical and Automation Engineering, the Chinese University of Hong Kong, Hong Kong, China (e-mail: yijieyang@cuhk.edu.hk, yuechen@mae.cuhk.edu.hk).} 
\thanks{X. Weng is with the Guanghua School of Management, Peking University, China (email: wengxi125@gsm.pku.edu.cn)}
}

\markboth{}
{Shell \MakeLowercase{\textit{et al.}}: Bare Demo of IEEEtran.cls for IEEE Journals}

\maketitle

\begin{abstract}
The rapid growth of AI workloads is substantially increasing data center electricity demand and carbon emissions, motivating the development of carbon-aware scheduling methods. However, effective scheduling is challenging because renewable generation and AI workloads are subject to forecast errors, while training and inference workloads exhibit heterogeneity in computational characteristics.  
This paper proposes a contextual robust optimization framework for AI data center operation. The proposed model explicitly captures the heterogeneous computational characteristics of AI training and inference workloads. To deal with renewable generation and workload forecast errors, we develop loss-based uncertainty learning models that directly map contextual features to covariate-dependent uncertainty sets. The resulting contextual joint chance-constrained scheduling problem is reformulated into a tractable robust optimization problem, and a calibration algorithm is developed to provide finite-sample probabilistic feasibility guarantees for multiple joint chance constraints. Numerical experiments based on real-world AI workload traces and renewable generation data show that the proposed method reduces operating costs by an average of 5.57\% compared to benchmark methods while maintaining reliable feasibility and strong computational scalability.
\end{abstract}


\section{Introduction}
The rapid growth of generative AI is substantially increasing the electricity demand of data centers. In 2025, data centers consumed approximately 485 TWh, representing about 1.7\% of global electricity use, and this demand is projected to reach approximately 950 TWh by 2030, accounting for around 3\% of global electricity demand \cite{IEA2026EnergyAI}.
Such large-scale energy consumption can lead to considerable greenhouse gas (GHG) emissions. In response, major technology companies have increasingly invested in renewable energy procurement and on-site generation. Amazon, Meta, and Google have collectively supported 22 GW of renewable energy capacity to advance their net-zero emission targets~\cite{acun2023carbon}. At the same time, due to the variability and intermittency of renewable sources, data centers need scheduling methods that can effectively align computing demand with renewable energy availability. 

Cost-aware and carbon-aware data center scheduling has received increasing attention in recent studies. Google classified workloads into inflexible and flexible ones and shifted flexible workloads to periods with sufficient renewable energy \cite{radovanovic2022carbon}. A multi-objective carbon-aware workload scheduling strategy was proposed in \cite{wang2025multi} to minimize carbon emissions while maximizing quality of service (QoS) under time-varying grid carbon intensity. The trade-off between operational and embodied carbon for 24/7 carbon-free data center operation was investigated in \cite{acun2023carbon}. In addition, a game-based model was introduced in \cite{wu2025game} to optimize workload allocation across spatial and temporal dimensions and reduce data center operating costs. These studies demonstrate the potential of renewable energy integration, energy storage, and workload shifting to reduce total operating costs and carbon emissions. However, most existing studies focus on general-purpose data centers and model computing demand using only simplified categories of flexible and inflexible workloads, without explicitly capturing the distinct characteristics of AI workloads. In particular, AI training jobs are heterogeneous in hardware type, model architecture, and delay tolerance, resulting in varying spatio-temporal scheduling flexibility and power consumption patterns. Inference requests exhibit heterogeneity across tasks, input and output lengths, and serving configurations, leading to varying latency requirements and energy profiles. 

Uncertainty in scheduling can significantly affect the planning and operation of energy systems and therefore cannot be overlooked in data center operations. Existing studies have addressed data center scheduling under uncertainty using stochastic optimization (SO) \cite{fan2025stochastic,wang2020stochastic}, robust optimization (RO) \cite{seyyedi2024application}, and distributionally robust optimization (DRO) \cite{CarbonAwareHall}. SO typically assumes that uncertain parameters follow a predefined probability distribution, which may be difficult to obtain accurately in practice. RO relaxes this requirement by protecting decisions against all realizations within a prescribed uncertainty set, but this worst-case protection often leads to conservative solutions. DRO further generalizes this idea by optimizing against the worst-case probability distribution within an ambiguity set constructed from empirical data.

Specifically, chance-constrained programming (CCP) is an SO approach, which allows constraints to be violated with a prescribed probability and thus provides a trade-off between operational reliability and conservatism. However, classical CCP usually requires prior knowledge of the underlying probability distribution and does not provide finite-sample probabilistic feasibility guarantees under distributional misspecification. To address these limitations, distributionally robust chance-constrained (DRCC) methods have been introduced into data center operation \cite{chen2025spatial,li2023distributionally}.  Nevertheless, solving complex CCP and DRCC formulations remains challenging due to their non-convexity, non-differentiability, and computational complexity, especially in large-scale data center scheduling problems.


In addition, most existing uncertainty-aware optimization methods do not explicitly incorporate contextual information, such as historical observations, weather conditions, and temporal features, when characterizing the uncertainty of prediction error. As a result, the joint observation of covariates and prediction error is ignored, and a static uncertainty set is applied to cover heterogeneous operating conditions. However, forecast errors are often context-dependent and exhibit significant variability across different forecast values \cite{wang2018conditional}. This mismatch can introduce unnecessary conservatism and increase operating costs in the downstream optimization problem.  



Recent studies have tried to incorporate contextual information into uncertainty modeling through machine learning (ML) and data-driven methods. For example, a sample average approximation (SAA)-based approach is proposed in \cite{kannan2024residuals}, which leverages covariates to construct ambiguity sets. Deep learning-based RO methods \cite{goerigk2023data,chenreddy2022data} learn uncertainty sets using neural networks and then reformulate them as mixed-integer programs (MIP). Similarly, an ML-MIP-RO method proposed in \cite{bertsimas2025constructing} relies on mixed-integer formulations to construct covariate-dependent uncertainty bounds. However, these approaches can be computationally expensive for large-scale systems and often rely on restrictive assumptions about prediction models \cite{chenreddy2022data}. Moreover, existing covariate-dependent methods mainly address uncertainty in the objective function \cite{bertsimas2025data} and do not explicitly provide probabilistic feasibility guarantees on joint chance constraints (JCCs).

Motivated by the aforementioned gaps, this paper proposes a contextual robust optimization method for carbon-aware AI data center workload scheduling. 
The main contributions of this paper are twofold:

1) \textit{Contextual robust optimization model for carbon-aware AI data center scheduling:}
We formulate a carbon-aware scheduling model for geographically distributed AI data centers under renewable generation and workload uncertainties. The proposed model explicitly captures the heterogeneous computational characteristics of AI training and inference workloads. Compared with general data center scheduling models, the proposed formulation provides a more accurate representation of AI computing demand and its operational flexibility. The resulting scheduling problem is further formulated as a contextual CCP with multiple JCCs, enabling reliable operation under forecast uncertainty.

2) \textit{Loss-based contextual uncertainty sets with probabilistic feasibility guarantees:}
    We develop loss-based uncertainty learning models to construct covariate-dependent uncertainty sets for renewable generation and workload uncertainties. Unlike traditional predict-then-optimize methods for RO, which first train a point prediction model and then construct uncertainty structures (e.g., uncertainty sets or distributions) from historical residuals, the proposed method directly learns the uncertainty sets from data. This enables the resulting uncertainty sets to capture the dependence of forecast errors on covariates and operating conditions. We further develop a calibration procedure that provides a finite-sample probabilistic feasibility guarantee for JCCs. Numerical results demonstrate that the proposed method reduces operational cost and carbon emissions by at least 5.57\% and 3.41\%, respectively, compared with benchmark methods.

\section{Model Formulation} \label{formulation}

We investigate the joint optimization of computing and energy resources for the low-carbon operation of geographically distributed data centers. Specifically, we consider a system comprising $N$ data centers interconnected through a common network, as illustrated in Fig.~\ref{fig:system_framework}. Each DC can purchase electricity from the main grid, utilize on-site renewable generation, and operate an energy storage system (ESS). 
\begin{figure}
    \centering
    \includegraphics[width=1\linewidth]{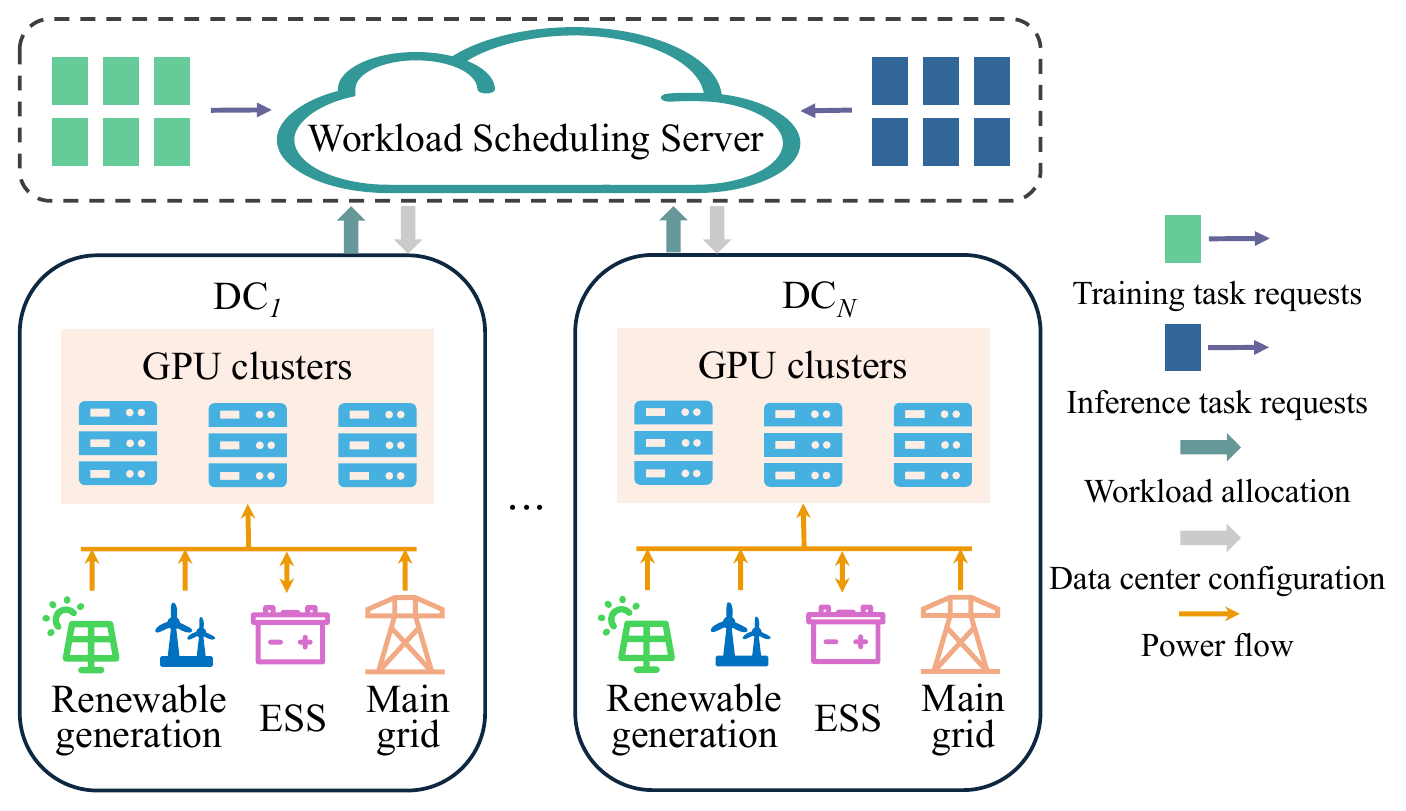}
    \caption{The framework of coordinated geographically distributed AI data centers}
    \label{fig:system_framework}
\end{figure}
\subsection{AI Workload Model}
In an AI-centric data center, workloads primarily consist of \emph{training} and \emph{inference} tasks. Training workloads are delay-tolerant with spatio-temporal flexibility, while inference workloads are latency-sensitive but spatially flexible via request routing across data centers. Based on these characteristics, workload allocation and energy procurement are co-optimized to minimize the total electricity and carbon costs while satisfying service-level and operational constraints.

\subsubsection{AI Training Workloads}

In practice, geographically distributed data center operations are large-scale, making explicit job-level scheduling computationally intractable. Therefore, in real-world data centers, such as Google’s clusters, load shifting is implemented through predefined cluster-level \emph{Virtual Capacity Curves} (VCCs) \cite{radovanovic2022carbon}, which specify hourly aggregate compute limits for each data center, thereby constraining the resources available to the real-time scheduler. Following this practice, we model AI training demand at an aggregate level.

Based on the model in \cite{CarbonAwareHall}, we classify training workloads according to their spatial and temporal flexibility. Each class $c \in \mathcal{C}$ is characterized by a temporal flexibility parameter $h_c$, which represents the maximum delay that workloads of class $c$ can tolerate, and by a set of data centers $\mathcal{D}_c \subseteq \mathcal{D}$, which specifies the locations at which workloads of class $c$ can be allocated. The training workload scheduling problem can be modeled as follows:
\begin{subequations}
\label{eq:1}
\begin{align}
\small
&\sum\nolimits_{d \in \mathcal{D}_c} \sum\nolimits_{t \in \mathcal{T}_{k}^{(c)}} y_{k,c,t,d}^{\mathrm{tra}} = 1,
\quad \forall k \in \mathcal{T},\ c \in \mathcal{C},
\label{eq:train_alloc_balance}\\
&y_{k,c,t,d}^{\mathrm{tra}} = 0,
\quad \forall k \in \mathcal{T},\  c \in \mathcal{C},\  d \in \mathcal{D} \setminus \mathcal{D}_c,\ t \in \mathcal{T},
\label{eq:train_alloc_space}\\
&y_{k,c,t,d}^{\mathrm{tra}} = 0,
\quad \forall k \in \mathcal{T},\ c \in \mathcal{C},\  d \in \mathcal{D},\ t \in \mathcal{T} \setminus \mathcal{T}_{k}^{(c)}.
\label{eq:train_alloc_time} 
\end{align}
\end{subequations}
where $y_{k,c,t,d}^{\mathrm{tra}} \in [0,1]$ is the proportion of the workload of class $c$ arriving at time slot $k$ that is executed at data center $d$ during time slot $t$.
$\mathcal{T}_{k}^{(c)} := \left\{ t \in \mathcal{T} \,\middle|\, k \le t \le \min(k+h_c,\,T) \right\}$ is the acceptable execution time slots. Eq. \eqref{eq:train_alloc_balance} ensures that all arriving workload is fully allocated, while Eqs. \eqref{eq:train_alloc_space} and \eqref{eq:train_alloc_time} enforce spatial and temporal feasibility, respectively.

The amount of class $c$ training workload executed at data center $d$ during time slot $t$ is then given by
\begin{align}
L_{d,c,t}^{\mathrm{tra}}
=
\sum\nolimits_{k \in \mathcal{T}}
y_{k,c,t,d}^{\mathrm{tra}} \lambda_{c,k}^{\mathrm{tra}},
\quad
\forall d \in \mathcal{D},\ c \in \mathcal{C},\ t \in \mathcal{T}.
\label{eq:train_load}
\end{align}
where $\lambda_{c,k}^{\mathrm{tra}}$ is the amount of class $c$ training workload arriving at time slot $k \in \mathcal{K}$. We measure $\lambda_{c,k}^{\mathrm{tra}}$ in GPU-equivalent units, which means that one unit corresponds to the use of one GPU during a single time slot of duration $\Delta_t$ (e.g., the workload is measured in GPU-hours if $\Delta_t = 1$ hour).

Let $G_{d,c,t}^{\mathrm{tra}}$
denote the GPU capacity assigned to training workloads of class $c$ at data center $d$ during time slot $t$. The corresponding capacity constraints are
\begin{subequations}
\begin{align}
& L_{d,c,t}^{\mathrm{tra}} \le \Delta_t G_{d,c,t}^{\mathrm{tra}},
\quad \forall d \in \mathcal{D},\ \forall c \in \mathcal{C},\ \forall t \in \mathcal{T},
\label{eq:train_gpu_match}\\
&\sum\nolimits_{c \in \mathcal{C}} G_{d,c,t}^{\mathrm{tra}} \le G_{d,\text{max}}^{\mathrm{tra}},
\quad \forall d \in \mathcal{D},\ \forall t \in \mathcal{T}.
\label{eq:train_gpu_cap}
\end{align}
\end{subequations}
Eq. \eqref{eq:train_gpu_match} ensures that the allocated GPU capacity is sufficient to serve the scheduled training workload of each class, while \eqref{eq:train_gpu_cap} enforces that the aggregate allocation does not exceed the GPU capacity of data center $d$. This formulation is applicable to both homogeneous and heterogeneous GPU settings. For heterogeneous GPUs, class $c$ can be defined to represent both workload type and GPU type, and the corresponding GPU capacity limits can be adjusted accordingly.

\subsubsection{AI Inference Workloads}

AI inference workloads are latency-sensitive and must be served upon arrival, but can be routed to DCs in different regions. Inference tasks are often classified by task (e.g., coding and conversation) and input and output lengths \cite{stojkovic2025dynamollm}. 
Let $i \in \mathcal{I}$ denote an inference service class, $d \in \mathcal{D}$ a data center, and $p \in \mathcal{P}_i$ a serving configuration, such as GPU frequency and tensor-parallelism (TP) degree \cite{zhong2024distserve}. 
Let $x_{i,d,p,t}$ denote the inference workload of service $i$ assigned to data center $d$ under serving configuration $p$ during time slot $t$, and $\lambda^{\mathrm{inf}}_{i,t}$ denote the corresponding arrival rate. The incoming workload must be fully assigned:
\begin{subequations}
\begin{align}
&\sum_{d \in \mathcal{D}} \sum_{p \in \mathcal{P}_i} x_{i,d,p,t}
\geq \lambda^{\mathrm{inf}}_{i,t},
\quad \forall i \in \mathcal{I},\; t \in \mathcal{T}, \\
& x_{i,d,p,t} \geq 0,
\quad \forall i, d, p, t.
\label{eq:inf_balance}
\end{align}
\end{subequations}
The end-to-end response latency $T^{\mathrm{res}}_{i,d,p,t}$ consists of queueing delay $T^{\mathrm{que}}_{i,d,p,t}$ and model execution delay $T^{\mathrm{exe}}_{i,d,p}$ as follows: 
\begin{equation}
T^{\mathrm{res}}_{i,d,p,t}
=
T^{\mathrm{que}}_{i,d,p,t}
+
T^{\mathrm{exe}}_{i,d,p},
\quad \forall i,d,p,t .
\label{eq:response_latency}
\end{equation}
The queueing delay is modeled as
\begin{subequations}
\begin{align}
&T^{\mathrm{que}}_{i,d,p,t}
=
\frac{1}{\mu_{i,d,p}-x_{i,d,p,t}/m_{i,d,p,t}},
\quad \forall i,d,p,t, 
\label{eq:inf_queue_delay} \\
&x_{i,d,p,t} \leq m_{i,d,p,t}\mu_{i,d,p},
\quad \forall i,d,p,t .
\label{eq:inf_queue_stability}
\end{align}
\end{subequations}
where $m_{i,d,p,t}$ denotes the number of active inference instances, and  $\mu_{i,d,p}$ denotes the rated service rate of each active instance.  Each active inference instance is approximated as an independent queue. 

Generally, the execution time of the inference task under different configurations, $T^{\mathrm{exe}}_{i,d,p}$, can be obtained through offline profiling and empirical measurements. For LLM inference, execution consists of the \textit{prefill} and \textit{decode} phases. The prefill phase processes the input prompt and constructs the KV cache, whereas the decode phase generates output tokens sequentially. Accordingly, the relevant latency SLOs include Time To First Token (TTFT) and Time Between Tokens (TBT) \cite{stojkovic2025dynamollm}. The TTFT latency and total generation latency are modeled as:
\begin{subequations}
\begin{align}
&T^{\mathrm{TTFT}}_{i,d,p,t}
=
T^{\mathrm{que}}_{i,d,p,t}
+
T^{\mathrm{pre}}_{i,d,p},
\quad \forall i,d,p,t ,\label{eq:first_token_latency}\\
&T^{\mathrm{res}}_{i,d,p,t}
=
T^{\mathrm{TTFT}}_{i,d,p,t}
+
L^{\mathrm{out}}_i T^{\mathrm{TBT}}_{i,d,p},
\quad \forall i,d,p,t .
\label{eq:total_generation_latency}
\end{align}
\end{subequations}
where $T^{\mathrm{pre}}_{i,d,p}$ denotes the prefill computation latency of LLM service $i$ under serving configuration $p$ and $T^{\mathrm{TBT}}_{i,d,p}$ denotes the average time between output tokens. They can be obtained through offline profiling of the LLM serving system under different configurations. $L^{\mathrm{out}}_i$ denotes the expected output sequence length of service class $i$. Since inference requests are classified according to their input and output sequence lengths, $L^{\mathrm{out}}_i$ can be estimated from historical traces as the maximum tokens for requests belonging to class $i$. 

To satisfy latency requirements, each active serving configuration must meet the corresponding latency service level objectives (SLOs) as follows:
\begin{equation}
T^{\mathrm{res}}_{i,d,p,t}
\leq
T^{\mathrm{SLO,res}}_i,
\quad \forall i,d,p,t .
\label{eq:response_sla}
\end{equation}

Additional latency requirements are imposed on the TTFT latency and TBT latency for LLM inference services:
\begin{subequations}
\begin{align}
&T^{\mathrm{TTFT}}_{i,d,p,t}
\leq
T^{\mathrm{SLO,TTFT}}_i,
\quad \forall i,d,p,t ,
\label{eq:first_token_sla} \\
&T^{\mathrm{TBT}}_{i,d,p}
\leq
T^{\mathrm{SLO,TBT}}_i,
\quad \forall i,d,p.
\label{eq:tbt_sla}
\end{align}
\end{subequations}
The serving configuration $p$ can represent the TP degree, e.g., $\mathcal{P}_i=\{2,4,8\}$, where a larger TP degree allocates more GPUs to one inference instance and may reduce computation latency at the cost of increased power consumption. Let $g_{i,p}$ denote the number of GPUs required by one active instance of service $i$ under serving configuration $p$, and $G^{\max}_{d,t}$ denote the total available GPUs at data center $d$ during time slot $t$. 
The GPU capacity constraint is modeled as
\begin{equation}
\sum_{i \in \mathcal{I}} \sum_{p \in \mathcal{P}_i}
g_{i,p} m_{i,d,p,t}
\leq
G^{\max}_{d,t},
\quad \forall d \in \mathcal{D},\; t \in \mathcal{T}.
\label{eq:gpu_capacity}
\end{equation}

\subsection{AI Data Center Power Consumption Model}

The total power consumption of an AI data center consists of the power consumed by IT equipment that processes workloads and the auxiliary services, such as cooling. Since GPU power dominates the workload-dependent component of IT power, the total data center power consumption $P_{d,t}^{\mathrm{DC}}$ is modeled as follows:
\begin{subequations}
\begin{align}
& P_{d,t}^{\mathrm{DC}} = \mathrm{PUE}_d \, P_{d,t}^{\mathrm{IT}},
~ \forall d \in \mathcal{D},\ t \in \mathcal{T}, \\
& P_{d,t}^{\mathrm{IT}} = P_{d}^{\mathrm{base,IT}} + \alpha_d \left( P_{d,t}^{\mathrm{tra}} + P_{d,t}^{\mathrm{inf}} \right),
~ \forall d \in \mathcal{D},\ t \in \mathcal{T},
\end{align}
\end{subequations}
where $\mathrm{PUE}_d$ denotes the Power Usage Effectiveness factor of data center $d$. It is defined as the ratio of total facility power to IT power. $P_{d,t}^{\mathrm{IT}}$ denotes the total IT power, and $\alpha_d$ is the GPU-to-IT power scaling factor \cite{reddy2025ai} that accounts for the additional IT power consumption of other components, including CPUs, memory, storage, and networking devices.

Training workloads are executed on GPUs with high but relatively stable utilization. The GPU power consumption for training is given by
\begin{align}
\small
&P_{d,t}^{\mathrm{tra}}
=
\sum\nolimits_{c \in \mathcal{C}}
G_{d,c,t}^{\mathrm{tra}}
P_{d,c}^{\mathrm{p,tra}}
u_{d,c,t}^{\mathrm{tra}}, \forall d \in \mathcal{D},\ t \in \mathcal{T},
\label{eq:train_power}
\end{align}
where $P_{d,c}^{\mathrm{p,tra}}$ is the peak GPU power for training workload class $c$ at data center $d$, and $u_{d,c,t}^{\mathrm{tra}} $ denotes the average utilization rate of GPUs of training class $c$ workloads.

Since inference tasks are deployed on different instances, each instance can be seen as a server.  The total inference GPU power at data center $d$ is
\begin{subequations}
\begin{align}
&P_{d,t}^{\mathrm{inf}}
=
\sum\nolimits_{i \in \mathcal{I},p \in \mathcal{P}_i} m_{i,d,p,t}
\Big[
P_{i,d,p}^{\mathrm{i,inf}}+
\big(P_{i,d,p}^{\mathrm{p,inf}} - P_{i,d,p}^{\mathrm{i,inf}}\big)
\rho_{i,d,p,t}^{\mathrm{inf}}
\Big], \notag \\
&\qquad \qquad \qquad \qquad \qquad \qquad \qquad \forall d \in \mathcal{D}, t \in \mathcal{T},
\label{eq:inf_power} \\
&\rho_{i,d,p,t}^{\mathrm{inf}}
=
\frac{x_{i,d,p,t}}{m_{i,d,p,t}\mu_{i,d,p}},
\qquad
\forall i \in \mathcal{I},\, d \in \mathcal{D},\, p \in \mathcal{P}_i,\, t \in \mathcal{T}.
\label{eq:inf_util}
\end{align}
\end{subequations}
where $\mu_{i,d,p}$ is the rated service rate of one active server. $P_{i,d,p}^{\mathrm{i,inf}}$ and $P_{i,d,p}^{\mathrm{p,inf}}$ represent the idle and peak power of one active server, respectively. $\rho_{i,d,p,t}^{\mathrm{inf}}$ denotes the average utilization of active serving instances.

\subsection{Energy Storage System Model}

Each data center is equipped with an energy storage system (ESS) to shift electricity consumption across time. Let $E_{d,t}^{\mathrm{ess}}$, $P_{d,t}^{\mathrm{ch}}$, and $P_{d,t}^{\mathrm{dis}}$ denote the stored energy, charging power, and discharging power of the ESS at data center $d$ and time slot $t$, respectively. The ESS operation is modeled as
\begin{subequations}
\label{eq:ess}
\begin{align}
&SOC_{d,t} = E_{d,t}^{\mathrm{ess}}/{E_d^{\max}}, \forall d \in \mathcal{D},\ t \in \mathcal{T},
\label{eq:ess_soc}\\
&E_{d,t}^{\mathrm{ess}}
=
E_{d,t-1}^{\mathrm{ess}}
+
\left(
\eta_d^{\mathrm{ch}} P_{d,t}^{\mathrm{ch}}
-
P_{d,t}^{\mathrm{dis}}/{\eta_d^{\mathrm{dis}}}
\right)\Delta, \forall d \in \mathcal{D},\ t \in \mathcal{T},
\label{eq:ess_dynamics}\\
&E_{d,1}^{\mathrm{ess}} = E_{d,|\mathcal{T}|}^{\mathrm{ess}},
 \forall d \in \mathcal{D},
\label{eq:ess_cyclic}\\
&SOC_d^{\min} \le SOC_{d,t} \le SOC_d^{\max},
 \forall d \in \mathcal{D},\ t \in \mathcal{T},
\label{eq:ess_soc_bound}\\
&0 \le P_{d,t}^{\mathrm{ch}} \le z_{d,t}^{\mathrm{ess}} P_d^{\max,\mathrm{ch}},
 \forall d \in \mathcal{D},\ t \in \mathcal{T},
\label{eq:ess_ch_limit}\\
&0 \le P_{d,t}^{\mathrm{dis}} \le (1-z_{d,t}^{\mathrm{ess}}) P_d^{\max,\mathrm{dis}},
 \forall d \in \mathcal{D},\ t \in \mathcal{T},
\label{eq:ess_dis_limit}\\
&z_{d,t}^{\mathrm{ess}} \in \{0,1\},
 \forall d \in \mathcal{D},\ t \in \mathcal{T},
\label{eq:ess_binary}
\end{align}
\end{subequations}
where $E_d^{\max}$ is the BESS energy capacity, $SOC_d^{\min}$ and $SOC_d^{\max}$ are the minimum and maximum state-of-charge limits, $P_d^{\max,\mathrm{ch}}$ and $P_d^{\max,\mathrm{dis}}$ are the charging and discharging power limits, $E_{d,1}^{\mathrm{ess}}$ denotes the amount of initial stored energy, and $\eta_d^{\mathrm{ch}}$ and $\eta_d^{\mathrm{dis}}$ are the charging and discharging efficiencies, respectively.

\subsection{Joint CCP Formulation}
To enable carbon-aware load shifting, future renewable generation and workload arrivals need to be predicted in advance. Although forecasts are available, prediction errors are unavoidable. To ensure reliable operation under such uncertainty, constraints involving predicted variables are formulated as chance constraints:
\begin{subequations}
 \begin{align}
&\resizebox{0.85\linewidth}{!}{$\displaystyle
\mathbb{P}\left(
P_{d,t}^{\mathrm{renew}} + P_{d,t}^{\mathrm{grid}} + P_{d,t}^{\mathrm{dis}}
\ge P_{d,t}^{\mathrm{DC}} + P_{d,t}^{\mathrm{ch}}
\right) \ge 1-\epsilon_{d,t},\forall d, t 
$}
\label{cc1_new}\\
&\resizebox{0.85\linewidth}{!}{$\displaystyle
\mathbb{P} \left(
\Delta_t G_{d,c,t}^{\mathrm{tra}} \ge \sum\nolimits_{k} y_{k,c,t,d}^{\mathrm{tra}} \lambda_{c,k}^{\mathrm{tra}},
\ \forall t
\right) \ge 1-\epsilon_{d,c}, \forall d, c. \label{cc2_new}
$}\\
&\resizebox{0.8\linewidth}{!}{$\displaystyle
\mathbb{P} \left(
\sum\nolimits_{d} \sum\nolimits_{p} x_{i,d,p,t}
\ge \lambda_{i,t}^{\mathrm{inf}},
\ \forall t
\right) \ge 1-\epsilon_{i}, \forall i. \label{cc3_new}
$}
\end{align}
\end{subequations}
Eq. \eqref{cc1_new} ensures that the power supply is sufficient to meet the data-center power demand. Eqs. \eqref{cc2_new} and \eqref{cc3_new} require that the allocated computing capacity can satisfy the training and inference workload demands, respectively. The power balance constraints are modeled as individual chance constraints (ICCs) to enforce per-slot adequacy, while the workload capacity constraints are formulated as stricter joint chance constraints (JCCs) to ensure service feasibility over the entire scheduling horizon.

The overall problem is written as:
\begin{subequations}
\begin{align}
\min &  \quad
\sum_{t \in \mathcal{T}} \sum_{d \in \mathcal{D}}
\alpha_{d,t}^{e} P_{d,t}^{\mathrm{grid}}\Delta_t
+
\sum_{d \in \mathcal{D}} \lambda_d^{c} C_d\Delta_t\\
\mathbf{s.t.} 
&\quad  C_d = \sum_{t \in \mathcal{T}} CI_{d,t} \, P_{d,t}^{\mathrm{grid}},
\quad \forall d \in \mathcal{D},\label{cb1_new}  \\
& \quad \eqref{eq:1} - \eqref{eq:ess}, \eqref{cc1_new} - \eqref{cc3_new},
\end{align}
\end{subequations}
where $\alpha_{d,t}^{e}$ denotes the electricity price at data center $d$ and time slot $t$, $\lambda_d^c$ denotes the carbon tax at data center $d$, and $CI_{d,t}$ denotes the carbon intensity of the main grid supplying data center $d$ at time $t$. The objective includes the total energy purchase cost and the carbon emission cost.

\section{Uncertainty Learning-based Contextual Robust Optimization}
In the conventional predict-then-optimize robust optimization framework, the first stage typically trains a point prediction model for uncertainty by minimizing the discrepancy between predicted values and observed realizations. In the second stage, prediction errors are calculated on a test set and used to construct uncertainty sets for subsequent robust optimization.
However, these uncertainty sets are usually constructed from historical prediction errors and are therefore independent of the contextual features used by the prediction model. Moreover, conventional methods often lack explicit probabilistic feasibility guarantees for constraints, particularly for complex JCCs. To address these limitations, this section first introduces an integrated uncertainty learning (UL) and robust optimization framework, as shown in Fig.~\ref{fig:methods_framework}. The proposed UL model directly maps contextual features to uncertainty sets. We then develop a calibration procedure to provide probabilistic feasibility guarantees for the JCCs.

\begin{figure}
    \centering
    \includegraphics[width=1\linewidth]{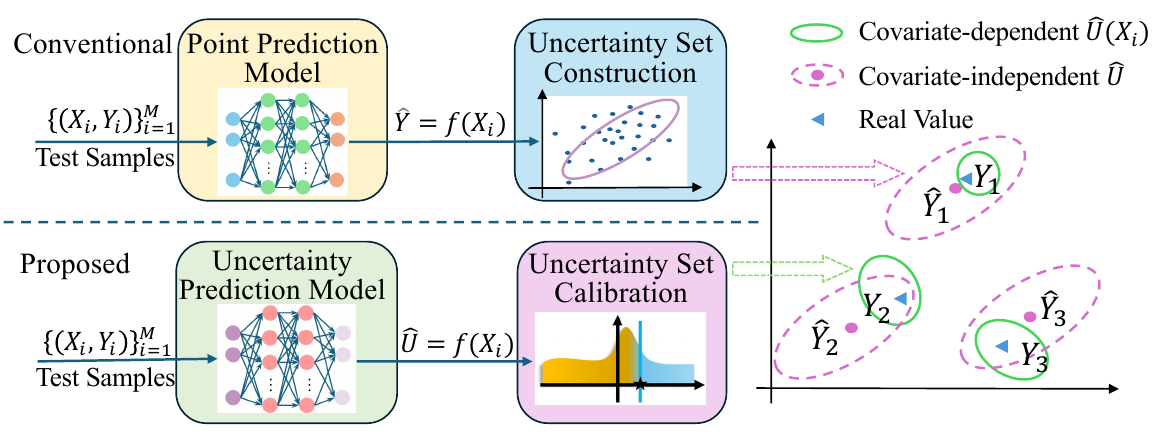}
    \caption{The difference between the proposed and conventional covariate-independent RO framework}
    \label{fig:methods_framework}
\end{figure}

\subsection{Contextual Robust Optimization}

Observing that all constraints inside the probability operators of Eqs.~\eqref{cc1_new}--\eqref{cc3_new} are linear, they can be expressed in a generic form:
\begin{align} g_{l,j}(\boldsymbol{x},\boldsymbol{\xi}_l) = \left( \boldsymbol{H}_{l,j}\boldsymbol{\xi}_l \right)^{\top} \boldsymbol{x} - b_{l,j}(\boldsymbol{x}), \quad \forall j\in [m_l],\ \forall l, \label{eq:generic_glj} \end{align} 
where $\boldsymbol{x}$ is the decision variable, $\boldsymbol{\xi}_l$ is the uncertain vector associated with the $l$-th chance constraint, $\boldsymbol{H}_{l,j}$ is the corresponding coefficient matrix, and $b_{l,j}(\boldsymbol{x})$ is affine in $\boldsymbol{x}$; and $l\in\mathcal{L}$ is the index of chance constraints, while $j \in [m_l]$ is the index of the constraints inside the probability operator. When $m_l=1$, the $l$-th constraint is an ICC, such as Eq.~\eqref{cc1_new}; otherwise, it is a JCC, such as Eq.~\eqref{cc2_new} and Eq.~\eqref{cc3_new}. We define a new function $h_l(\boldsymbol{x},\boldsymbol{\xi}_l)$ as:
\begin{align} h_l(\boldsymbol{x},\boldsymbol{\xi}_l) = \max_{j\in [m_l]} g_{l,j}(\boldsymbol{x},\boldsymbol{\xi}_l). \label{eq:hl_definition} \end{align} Then the overall CCP with multiple JCCs can be written as 
\begin{subequations} \label{eq:standard_ccp} \begin{align} &f^* = \min_{\boldsymbol{x}\in\mathcal{X}} f(\boldsymbol{x}) \label{eq:standard_ccp_obj} \\ \mathrm{s.t.}\quad & \mathbb{P}_{\boldsymbol{\xi}_l} \left( h_l(\boldsymbol{x},\boldsymbol{\xi}_l)\le 0 \right) \ge 1-\epsilon_l, \quad \forall l\in\mathcal{L}.\label{eq:standard_ccp_cc} \end{align} \end{subequations}

In practice, the uncertain vector
$\boldsymbol{\xi}_l$ is correlated with an observable covariate vector
$\boldsymbol{\psi}_l$. By conditioning on such contextual information, the
decision maker can exploit covariate-dependent uncertainty structures. The
contextual CCP with multiple JCCs is therefore formulated as
 \begin{subequations} \label{eq:standard_ccp} \begin{align} &\quad f^* = \min_{\boldsymbol{x}\in\mathcal{X}} f(\boldsymbol{x}) \label{eq:standard_ccp_obj} \\ \mathrm{s.t.}\quad & \mathbb{P}_{\xi_l \mid X} \left( h_l(\boldsymbol{x},\boldsymbol{\xi}_l)\le 0 | X=  \psi_l \right) \ge 1-\epsilon_l, \forall l.\label{eq:standard_c_ccp} \end{align} \end{subequations}
where $\mathbb{P}_{\boldsymbol{\xi}|X}$ is the conditional probability of $\xi$ given $X$. 
Since the conditional distribution of $\xi_l$ given the covariates is generally unavailable, directly solving the contextual CCP is difficult. We therefore replace each contextual JCC with a deterministic robust constraint defined over a covariate-dependent uncertainty set. Specifically, for a new context $\psi_l$, the uncertainty set $\mathcal{U}_l(\psi_l)$ is constructed from historical samples and contextual features, and the scheduling problem \eqref{eq:standard_ccp} is approximated by the following RO counterpart:
\begin{subequations} \label{eq:ro_multiple_jcc} \begin{align} f^* = & \min_{\boldsymbol{x}\in\mathcal{X}} \quad f(\boldsymbol{x}) \\ \mathrm{s.t.}\quad & h_{l}(\boldsymbol{x},\boldsymbol{\xi}_l)\le 0, \quad \forall \boldsymbol{\xi}_l\in\mathcal{U}_l(\psi_l), \  \forall l. \label{eq:jcc}\end{align} \end{subequations} 

\begin{theorem}\label{prop:robust_transformation_ccp} Any feasible solution of \eqref{eq:ro_multiple_jcc} is feasible for the chance constraints in \eqref{eq:standard_ccp} if the contextual uncertainty set $\mathcal{U}_l(\psi_l)$ for $\boldsymbol{\xi}_l$ in \eqref{eq:jcc} satisfies 
\begin{align}
\mathbb{P} \left( \boldsymbol{\xi}_l\in\mathcal{U}_l(\psi_l) \right) \ge 1-\epsilon_l . 
\end{align}
\end{theorem} 

The proof of Theorem \ref{prop:robust_transformation_ccp} is provided in the Appendix.

\begin{remark}
Traditional CCP generally requires distributional information of the uncertain parameters to evaluate or reformulate the probabilistic constraints. In contextual settings, this further requires characterizing the conditional distribution of the uncertainty given the observed covariates, which can be challenging when data are limited or high-dimensional. In contrast, the RO counterpart only relies on a prescribed uncertainty set. Theorem~\ref{prop:robust_transformation_ccp} shows that if the uncertainty set covers the true realization of $\xi_l$ with probability at least $1-\epsilon_l$, then any feasible solution of the RO counterpart also satisfies the corresponding JCCs. Therefore, the information of joint observations of $\psi_l$ and $\xi_l$ can be incorporated into the uncertainty representation without explicitly estimating the full conditional distribution.
\end{remark}

\subsection{Integrating Uncertainty into Prediction}

For a prediction model, a loss function is commonly used to quantify the discrepancy between predicted values and their corresponding realizations. Let $\hat{y}_n=f(\boldsymbol{X}_n)$ denote the prediction of the model for the input feature vector $\boldsymbol{X}_n$, and let $y_n$ denote the observed realization. Given a training dataset $\{(\boldsymbol{X}_n,y_n)\}_{n=1}^{N_1}$, the prediction model can be trained by minimizing an empirical loss function:
\begin{align}
    \min_{f} \ \frac{1}{N_1}\sum\nolimits_{n=1}^{N_1} 
    \ell\left(y_n, f(\boldsymbol{X}_n)\right),
    \label{eq:empirical_loss}
\end{align}
where $\ell(\cdot)$ measures the distance between the predicted value and the realized value. Two commonly used loss functions for regression tasks are the mean squared error (MSE) and the mean absolute error (MAE). Both losses quantify the discrepancy between predicted and realized values. A larger loss indicates a higher prediction error and thus suggests higher predictive uncertainty, whereas a smaller loss implies the prediction is closer to the observed outcome and reflects greater confidence in the model output.

Specifically, inspired by the data-driven uncertainty set construction framework in \cite{bertsimas2025data}, we construct covariate-dependent uncertainty sets by training prediction models whose loss functions are designed to characterize predictive uncertainty. Rather than first learning a point prediction and then constructing uncertainty sets from historical residuals, the proposed uncertainty-learning model directly maps contextual features to the parameters of the uncertainty set. In this work, we consider two realizations of this idea. For box uncertainty sets, an interval prediction model is trained using the interval quantile loss to predict lower and upper bounds. For ellipsoidal uncertainty sets, a mean-variance prediction model is trained to predict the contextual mean and variance, which define a loss-induced ellipsoidal uncertainty set. 

\subsubsection{Interval Quantile Loss}
For the box uncertainty set, uncertainty is characterized by the lower and upper bounds of the uncertain parameters. To learn these bounds, we propose the \emph{Interval Quantile Loss} (IQL):
\begin{align}
\small
\mathcal{L}_{\mathrm{IQL}}
&=
\frac{1}{N_1k}
\sum_{n=1}^{N_1}
\sum_{i=1}^{k}
\Big[
\ell_{\tau_{\mathrm{low}}}
\!\left(y^{i}_{n},\underline{y}^{i}(\boldsymbol{X}_{n})\right)
+\ell_{\tau_{\mathrm{high}}}
\!\left(y^{i}_{n},\overline{y}^{i}(\boldsymbol{X}_{n})\right)\nonumber
\\
&
+
\lambda_{\mathrm{w}}
\left(
\overline{y}^{i}(\boldsymbol{X}_{n})
-
\underline{y}^{i}(\boldsymbol{X}_{n})
\right)
\Big],
\label{eq:ciql_loss}
\end{align}
where $\lambda_{\mathrm{w}}\geq 0$  controls the width of the learned interval. 
$
\ell_{\tau}(y,\hat{y})
=
\max
\left\{
\tau(y-\hat{y}),
(\tau-1)(y-\hat{y})
\right\}
\label{eq:pinball_loss}
$ denotes the pinball loss. The quantile loss is used because the construction of a box uncertainty set requires estimating lower and upper conditional bounds of the uncertain parameters. While conventional losses such as MSE and MAE learn central predictions, the pinball loss directly estimates conditional quantiles. Hence, the lower- and upper-quantile losses learn context-dependent interval bounds, which can be directly embedded into the box uncertainty set. The width regularization term further penalizes excessively wide intervals to reduce conservatism. 

\subsubsection{Mean-Variance Loss}

Uncertainty can also be characterized by the mean and variance of the uncertain parameters. To this end, we propose the \emph{Mean-Variance Loss} (MVL):
\begin{align}
\small
\mathcal{L}_{\mathrm{MVL}}
=
\frac{1}{N_1k}
\sum_{n=1}^{N_1}
\sum_{i=1}^{k}
\Big[
\frac{
\left(
y^{i}_{n}
-
f_{i}(\boldsymbol{X}_{n})
\right)^{2}
}{
\sigma_{i}^{2}(\boldsymbol{X}_{n})
}
+
\log
\sigma_{i}^{2}(\boldsymbol{X}_{n})
\Big],
\label{eq:cmvl_loss}
\end{align}
where $\boldsymbol{\sigma}^{2}(\boldsymbol{X})$ denotes the variance. The first term measures the normalized prediction error. The second term regularizes the predicted variance. This formulation can yield an ellipsoidal uncertainty set centered at the predicted average value, with the shape determined by the learned variance. 

\subsection{Calibration for Probabilistic Feasibility Guarantees of JCCs}
\label{subsec:prob_guarantee}

After learning the contextual uncertainty sets, we further calibrate their sizes using an independent calibration dataset to obtain finite-sample coverage guarantees. Let
$\mathcal{D}_{\mathrm{cal}}=\{(\boldsymbol{X}_{n},\boldsymbol{y}_{n})\}_{n=1}^{N_2}$
denote the calibration dataset, and let $\epsilon$ denote the prescribed risk tolerance. The main idea is to compute a nonconformity score for each calibration sample, which measures how far the observed realization deviates from the predicted contextual uncertainty set. The empirical $(1-\epsilon)$-quantile of these scores is then used to enlarge the uncertainty set.
Specifically, the calibration parameter is the interval margin $q$ for the box uncertainty set and the radius $\rho$ for the ellipsoidal uncertainty set.

For the interval-based uncertainty learning model, the calibrated contextual box uncertainty set is defined as
\begin{align}
\mathcal{U}^{\mathrm{box}}(\boldsymbol{X})
=
\left\{
\boldsymbol{y}\in\mathbb{R}^{k}
\ \middle|\
\underline{y}^{i}(\boldsymbol{X})-q
\leq y^{i}
\leq
\overline{y}^{i}(\boldsymbol{X})+q,
\ \forall i
\right\}.
\label{eq:box_uncertainty_set}
\end{align}

For the mean-variance uncertainty learning model, the calibrated contextual ellipsoidal uncertainty set is defined as
\begin{align}
\mathcal{U}^{\mathrm{ell}}(\boldsymbol{X})
=
\left\{
\boldsymbol{y}\in\mathbb{R}^{k}
\ \middle|\
\sum_{i=1}^{k}
\frac{
\left(y^{i}-\mu_{i}(\boldsymbol{X})\right)^{2}
}{
\sigma_{i}^{2}(\boldsymbol{X})
}
\leq
\rho
\right\}.
\label{eq:mv_uncertainty_set}
\end{align}

Algorithms~\ref{alg:box_calibration} and~\ref{alg:ellipsoid_calibration} summarize the calibration procedures for the box and ellipsoidal uncertainty sets, respectively. In Algorithm~\ref{alg:box_calibration}, the nonconformity score is defined as the maximum violation of the predicted interval over all dimensions. The empirical $(1-\epsilon)$-quantile of these scores is then used as a common margin to expand both the lower and upper bounds. Since the score is computed jointly across all dimensions, the calibrated box uncertainty set provides joint coverage for the entire uncertain vector. Therefore, it can be directly embedded into the downstream robust optimization model to obtain probabilistic feasibility guarantees for JCCs.

Similarly, in Algorithm~\ref{alg:ellipsoid_calibration}, the nonconformity score is defined as the normalized squared deviation from the predicted contextual mean. The empirical $(1-\epsilon)$-quantile of these scores determines the calibrated ellipsoidal radius. Under the exchangeability assumption, the resulting calibrated uncertainty set satisfies finite-sample coverage and can be used as a tractable robust approximation of the contextual JCCs.

\begin{algorithm}[t]
\caption{Box Uncertainty Set Calibration}
\small
\label{alg:box_calibration}
\begin{algorithmic}[1]
\REQUIRE Trained interval prediction model $\hat{f}_{\mathrm{IQL}}$, calibration dataset
$\mathcal{D}_{\mathrm{cal}}=\{(\boldsymbol{X}_{n},\boldsymbol{y}_{n})\}_{n=1}^{N_2}$, risk tolerance $\epsilon$, new sample $\boldsymbol{X}_{N_2+1}$
\ENSURE Calibrated box uncertainty set $\hat{\mathcal{U}}_{\mathrm{box}}(\boldsymbol{X}_{N_2+1})$

\FOR{$n=1,\dots,N_2$}
    \STATE Predict
    $[\underline{\boldsymbol{y}}(\boldsymbol{X}_{n}),
    \overline{\boldsymbol{y}}(\boldsymbol{X}_{n})]
    \gets
    \hat{f}_{\mathrm{IQL}}(\boldsymbol{X}_{n})
    .$
    
    \STATE Compute the joint interval nonconformity score \\
    $
    s_{n}
    \gets
    \max_{i=1,\dots,k}
    \left\{
    \underline{y}^{i}(\boldsymbol{X}_{n})-y^{i}_{n},
    y^{i}_{n}-\overline{y}^{i}(\boldsymbol{X}_{n}),
    0
    \right\}.
    $
\ENDFOR

\STATE Sort $\{s_{n}\}_{n=1}^{N_2}$ in ascending order as
$s^{(1)}\leq \cdots \leq s^{(N_2)}$.

\STATE Set
$
r
\gets
\left\lceil
(N_2+1)(1-\epsilon)
\right\rceil
$, $
q
\gets
s^{(r)}.
$

\STATE Predict
$
[\underline{\boldsymbol{y}}(\boldsymbol{X}_{N_2+1}),
\overline{\boldsymbol{y}}(\boldsymbol{X}_{N_2+1})]
\gets
\hat{f}_{\mathrm{IQL}}(\boldsymbol{X}_{N_2+1}).
$

\STATE Construct $\hat{\mathcal{U}}_{\mathrm{box}}(\boldsymbol{X}_{N_2+1})$ according to Eq.~\eqref{eq:box_uncertainty_set}.

\STATE \textbf{return} $\hat{\mathcal{U}}_{\mathrm{box}}(\boldsymbol{X}_{N_2+1})$.
\end{algorithmic}
\end{algorithm}

\begin{algorithm}[t]
\caption{Ellipsoidal Uncertainty Set Calibration}
\small
\label{alg:ellipsoid_calibration}
\begin{algorithmic}[1]
\REQUIRE Trained mean--variance prediction model $\hat{f}_{\mathrm{MVL}}$, calibration dataset
$\mathcal{D}_{\mathrm{cal}}=\{(\boldsymbol{X}_{n},\boldsymbol{y}_{n})\}_{n=1}^{N_2}$, risk tolerance $\epsilon$, new sample $\boldsymbol{X}_{N_2+1}$
\ENSURE Calibrated ellipsoidal uncertainty set $\hat{\mathcal{U}}_{\mathrm{ell}}(\boldsymbol{X}_{N_2+1})$

\FOR{$n=1,\dots,N_2$}
    \STATE Predict
    $
    [\boldsymbol{\mu}(\boldsymbol{X}_{n}),
    \boldsymbol{\sigma}^{2}(\boldsymbol{X}_{n})]
    \gets
    \hat{f}_{\mathrm{MVL}}(\boldsymbol{X}_{n}).
    $
    
    \STATE Compute the ellipsoidal nonconformity score \\
    $
    s_{n}
    \gets
    \sum_{i=1}^{k}
    \frac{
    \left(y^{i}_{n}-\mu_{i}(\boldsymbol{X}_{n})\right)^{2}
    }{
    \sigma_{i}^{2}(\boldsymbol{X}_{n})
    }.
    $
\ENDFOR

\STATE Sort $\{s_{n}\}_{n=1}^{N_2}$ in ascending order as
$s^{(1)}\leq \cdots \leq s^{(N_2)}$.

\STATE Set
$
r
\gets
\left\lceil
(N_2+1)(1-\epsilon)
\right\rceil$,
$
\rho
\gets
s^{(r)}$.

\STATE Predict
$
[\boldsymbol{\mu}(\boldsymbol{X}_{N_2+1}),
\boldsymbol{\sigma}^{2}(\boldsymbol{X}_{N_2+1})]
\gets
\hat{f}_{\mathrm{MVL}}(\boldsymbol{X}_{N_2+1}).
$

\STATE Construct $\hat{\mathcal{U}}_{\mathrm{ell}}(\boldsymbol{X}_{N_2+1})$ according to Eq.~\eqref{eq:mv_uncertainty_set}.

\STATE \textbf{return} $\hat{\mathcal{U}}_{\mathrm{ell}}(\boldsymbol{X}_{N_2+1})$.
\end{algorithmic}
\end{algorithm}
The following theorem provides a finite-sample coverage guarantee for both Algorithm~\ref{alg:box_calibration} and Algorithm~\ref{alg:ellipsoid_calibration}. Specifically, the calibrated uncertainty set covers a new realization with probability at least $1-\epsilon$.

\begin{theorem}
\label{prop:coverage}
Let $\mathcal{D}_{\mathrm{cal}}=\{(\boldsymbol{X}_{n},\boldsymbol{y}_{n})\}_{n=1}^{N_2}$ be the calibration dataset, and let $(\boldsymbol{X}_{N_2+1},\boldsymbol{y}_{N_2+1})$ be a new sample from the same distribution. Denote the calibrated uncertainty sets obtained from Algorithm~\ref{alg:box_calibration} and Algorithm~\ref{alg:ellipsoid_calibration} by 
$\hat{\mathcal{U}}_{\mathrm{box}}(\boldsymbol{X}_{N_2+1})$ and 
$\hat{\mathcal{U}}_{\mathrm{ell}}(\boldsymbol{X}_{N_2+1})$, depending on the uncertainty set type. Then, for $e\in\{\mathrm{box},\mathrm{ell}\}$,
\begin{align}
\mathbb{P}
\left(
\boldsymbol{y}_{N_2+1}
\in
\hat{\mathcal{U}}_{e}
(\boldsymbol{X}_{N_2+1})
\right)
\geq
1-\epsilon .
\label{eq:prop_coverage}
\end{align}
\end{theorem}
The proof of Theorem \ref{prop:coverage} can be found in the Appendix.

After constructing the uncertainty set, the problem can be reformulated as a standard robust optimization problem with a box or an ellipsoidal uncertainty set \cite{ben2009robust}. The box-based formulation leads to a mixed-integer linear programming (MILP) problem, whereas the ellipsoidal formulation introduces second-order cone (SOC) constraints, resulting in a mixed-integer second-order cone programming (MISOCP) problem. Both formulations can be efficiently solved using commercial solvers such as Gurobi. The overall procedure is summarized in Algorithm~\ref{alg:ulro_dc_scheduling}. By combining Theorem~\ref{prop:robust_transformation_ccp} and Theorem~\ref{prop:coverage}, we establish in Proposition~\ref{prop_jcc} that the solution returned by Algorithm~\ref{alg:ulro_dc_scheduling} satisfies the prescribed probabilistic feasibility guarantee for the original JCCs.
\begin{algorithm}[t]
\caption{ULRO for AI Data Center Scheduling}
\small
\label{alg:ulro_dc_scheduling}
\begin{algorithmic}[1]
\REQUIRE
Historical renewable generation dataset $\mathcal{D}^{\mathrm{ren}}_d$;
historical workload dataset $\mathcal{D}^{\mathrm{tra}}_c$, $\mathcal{D}^{\mathrm{inf}}_i$;
risk tolerance $\epsilon$;
new context $\boldsymbol{X}_{n+1}$;
uncertainty-set type $e\in \{\mathrm{box},\mathrm{ell}\}$.
\ENSURE
Workload schedules and energy management decisions $\{y^{\mathrm{tra}}_{k,c,t,d},G^{\mathrm{tra}}_{d,c,t},x_{i,d,p,t}, P^{\mathrm{grid}}_{d,t},P^{\mathrm{ch}}_{d,t},P^{\mathrm{dis}}_{d,t}\}$.
\FOR{each dataset $\mathcal{D}^{s}\in\{\mathcal{D}^{\mathrm{ren}}_d, \mathcal{D}^{\mathrm{tra}}_c, \mathcal{D}^{\mathrm{inf}}_i \}$ }
\STATE Split $\mathcal{D}^s$ into a training set $\mathcal{D}_1^s$ and a calibration set $\mathcal{D}_2^s$.
\STATE Train the uncertainty-learning model on  $\mathcal{D}_1^s$ using the loss corresponding to $e$.
\ENDFOR
\FOR{each uncertain vector $\xi_l$ associated with JCC $l$}
\STATE Predict and calibrate the learned uncertainty set on the new sample $X_{n+1}^{\xi_l}$ using $\mathcal{D}_2^{\xi_l}$ and Algorithm~\ref{alg:box_calibration} if $e=\mathrm{box}$, or using Algorithm~\ref{alg:ellipsoid_calibration} if $e=\mathrm{ell}$.
\STATE Obtain the calibrated contextual uncertainty set
$\widehat{\mathcal{U}}^{\xi_l}_{e}(\boldsymbol{X}_{n+1})$.
\ENDFOR
\STATE Substitute the calibrated uncertainty sets into the RO formulation \eqref{eq:ro_multiple_jcc}.
\STATE  Reformulate and solve the resulting problem using standard RO techniques (e.g., column generation algorithm).
\RETURN $\{y^{\mathrm{tra}}_{k,c,t,d},G^{\mathrm{tra}}_{d,c,t},x_{i,d,p,t}, P^{\mathrm{grid}}_{d,t},P^{\mathrm{ch}}_{d,t},P^{\mathrm{dis}}_{d,t}\}$
\end{algorithmic}
\end{algorithm}

\begin{proposition}[Probabilistic feasibility guarantee for JCCs]
\label{prop_jcc} 
For the new sample $X_{n+1}$, suppose that each uncertainty set $\hat{\mathcal{U}}^{\xi_l}_{e}(X_{n+1})$, $e\in{\{\mathrm{box},\mathrm{ell}}\}$, is constructed and calibrated using Algorithms \ref{alg:box_calibration} and \ref{alg:ellipsoid_calibration}, respectively. Then any feasible solution of the RO formulation \eqref{eq:ro_multiple_jcc} satisfies the corresponding original JCCs in \eqref{eq:standard_c_ccp} at the prescribed confidence levels.
\end{proposition}

\section{Numerical studies}

\label{Numerical_studies}
\subsection{Data Description and Experimental Settings}

 We consider a set of three geographically distributed data centers. For AI training workloads, we use the data derived from the Alibaba cluster trace.\footnote{https://github.com/alibaba/clusterdata} This dataset includes the submission time, required GPU resources, and completion time of each AI training task, covering frameworks such as PyTorch and TensorFlow and multiple GPU types, including V100, T4, and P100. We classify the training workloads according to GPU type and quantify each workload class by its aggregated GPU time. The corresponding latency SLOs $h_c$ are set to 2, 4, and 7 hours, respectively. 
 For inference services, we employ the public Azure LLM serving traces,\footnote{https://github.com/Azure/AzurePublicDataset} which contain requests from both coding and conversation tasks. LLM serving latency and energy demand are strongly affected by the request structure. In particular, the input length mainly affects the prefill stage, whereas the output length governs the decoding stage.  This structure leads to heterogeneous latency and power requirements across requests. Following \cite{stojkovic2025dynamollm} and \cite{reddy2025ai}, we classify each class of inference requests into nine categories according to the combination of input and output token lengths, denoted by \{SS, SM, SL, MS, MM, ML, LS, LM, LL\}. Here, \textit{S}, \textit{M}, and \textit{L} represent short, medium, and long token sequences, with cutoffs defined by the 33rd, 66th, and 100th percentiles, respectively. The temporal request arrivals and the proportions of different request categories are reported in Figs.~\ref{fig:arriving} and~\ref{fig:classes}, which show clear periodicity. The class-wise latency and power parameters used in the experiments are taken from \cite{stojkovic2025dynamollm} and \cite{reddy2025ai}.
  Other data-center configuration parameters used in the experiments are adopted from \cite{yang2023two} and summarized in Table I. Each data center is also associated with an on-site renewable energy project. In the case study, the data centers in California and the Netherlands are equipped with PV resources, while a wind resource supplies the data center in Germany. Four years of hourly PV and wind generation profiles are simulated from public historical weather data following the methodology in \cite{staffell2016using}.\footnote{https://www.renewables.ninja} The corresponding carbon intensity is calculated based on the generation mix data provided by CAISO\footnote{https://www.caiso.com} and ENTSO-E\footnote{https://transparency.entsoe.eu}, while the carbon emission factors for each energy source are taken from \cite{maji2022carboncast}.

\begin{figure}[t]
\centering
\includegraphics[scale=0.32]{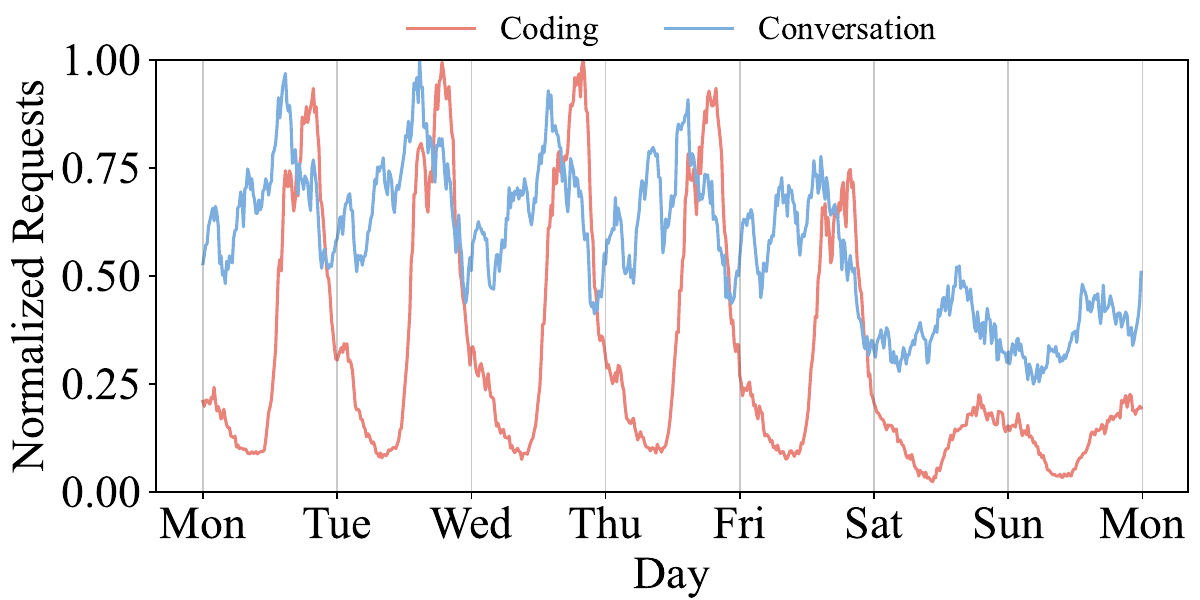}
\caption{Arrival patterns of coding and conversation tasks}
\label{fig:arriving}
\end{figure}

\begin{figure}[t]
\centering
\includegraphics[scale=0.3]{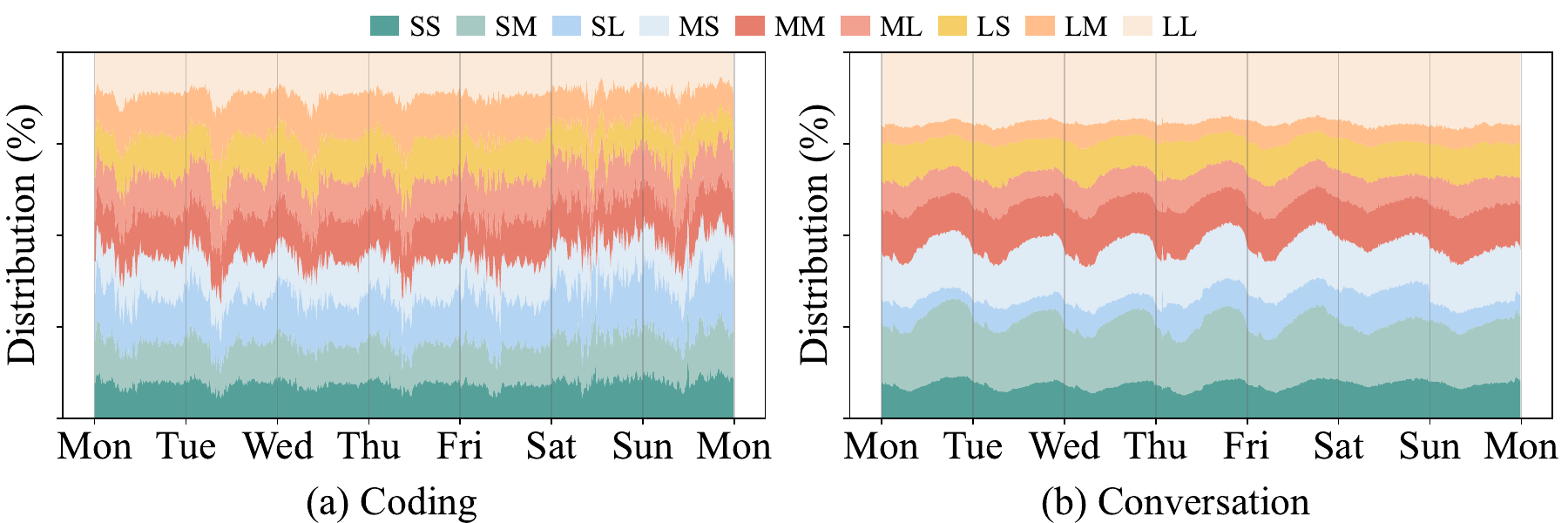}
\caption{Distribution of classes of coding and conversation tasks}
\label{fig:classes}
\end{figure}

\begin{table}[!t]
\centering
\footnotesize 
\caption{Relevant parameters for the data centers}
\label{tab:g1}
\centering
\resizebox{0.9\columnwidth}{!}{
\begin{tabular}{cccc}
\toprule
Parameters& Values & Parameters &Values\\
\midrule
$\text{PUE}_{d}$ & [1.2, 1.25, 1.3] & $\text{GPU}_{d}$ & [1.53, 1.53, 1.53] \\
$P^{\text{peak}}_{\text{tra}}$ (kW) & [0.25, 0.07, 0.25] & $G^{\text{tra}}_{\max}$ & [5000, 5000, 5000] \\
$G^{\text{inf}}_{\max}$ & [5500, 4600, 4700] & $P^{\text{grid}}_{\max}$ (kW) & [1470, 1550, 1580] \\
$h_c$ (h)& [2, 4, 7] & $M^{\text{ess}}_{\max}$ (kWh) & [400, 400, 400] \\
$P^{\text{ch}}_{\max}$ (kW) & [80, 80, 80] & $P^{\text{dis}}_{\max}$ (kW) & [80, 80, 80] \\
$\eta^{\text{ch}}$ & [0.95, 0.95, 0.95] & $\eta^{\text{dis}}$ & [0.95, 0.95, 0.95] \\
\bottomrule
\end{tabular}}
\end{table}

Day-ahead renewable generation and AI workload forecasts are obtained using multi-layer perceptron (MLP) models. For the baseline, the prediction models are trained using the MSE loss and use observations from the preceding 24 hours and temporal features, including day, week, and season of the year information. For renewable generation, weather-related features are also included, i.e., surface irradiance forecasts for PV generation and wind speeds for wind generation. The data are divided into training and testing sets with a 60\% and 40\% split. All simulations are conducted in Python on a Jetson AGX Orin module with an Arm Cortex-A78AE v8.2 CPU, an NVIDIA Ampere GPU, and 64 GB RAM. The optimization models are solved using Gurobi 11.0.3.


\subsection{Baselines}
We refer to the proposed method as ULRO and evaluate its performance against several widely used benchmarks for decision-making under uncertainty.

1) \textit{Classical Adaptive Robust Optimization (ARO)}:
ARO constructs covariate-independent uncertainty sets, whose size is calibrated to attain the prescribed empirical coverage \cite{ben2009robust}. We consider both box and ellipsoidal uncertainty sets, denoted by ARO-B and ARO-E, respectively.

2) \textit{Traditional Distributionally Robust Chance-Constrained Optimization (DRCC)}:
DRCC assumes that the unknown distribution belongs to a prescribed ambiguity set \cite{calafiore2006distributionally}. The joint chance constraints are decomposed into individual chance constraints using the Bonferroni approximation.

3) \textit{Residuals-based Distributionally Robust Optimization (ReDRO)}:
ReDRO is a covariate-dependent robust optimization method that constructs uncertainty sets from prediction residuals using a sample average approximation scheme \cite{kannan2024residuals}.

4) \textit{CVaR-based Data-Driven Distributionally Joint Chance-Constrained Optimization (CVaR-DJCC)}:
It builds a Wasserstein ambiguity set from historical data and approximates the resulting distributionally robust JCCs using Conditional Value-at-Risk (CVaR) for computational tractability \cite{zhai2022data}.

\begin{figure}[t]
\centering
\includegraphics[scale=0.35]{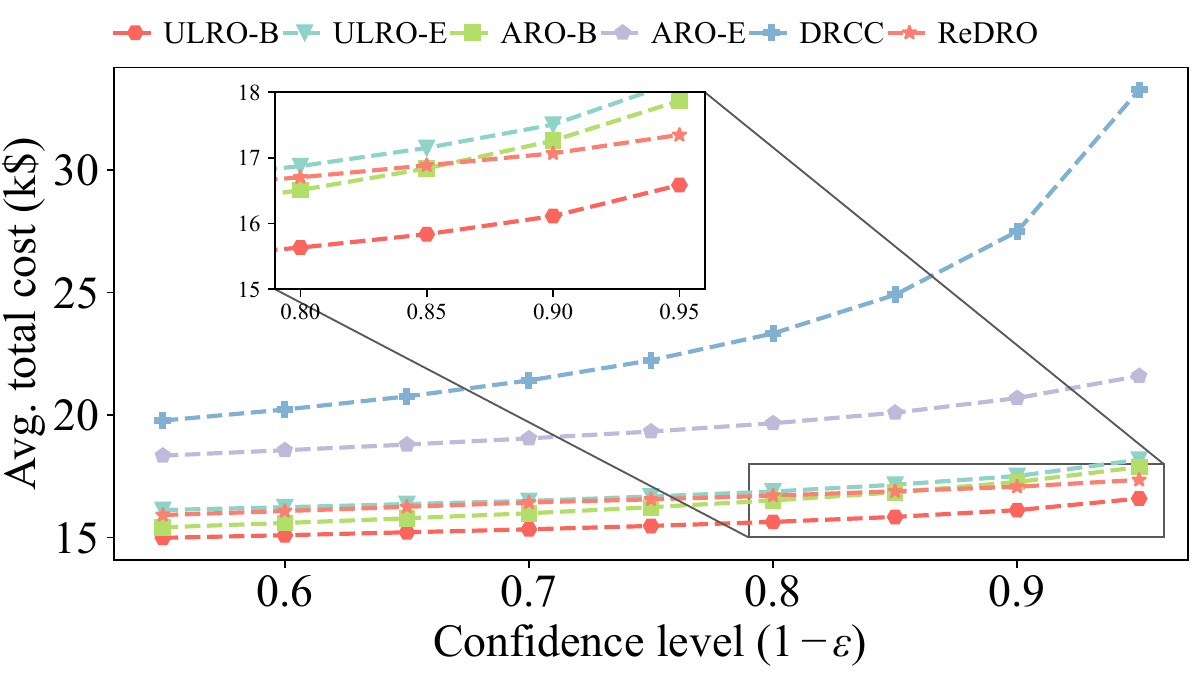}
\caption{Average total cost of the data center with different methods and coverage levels (k\$)}
\label{fig:cost_epison}
\end{figure}

\begin{figure}[t]
\centering
\includegraphics[scale=0.35]{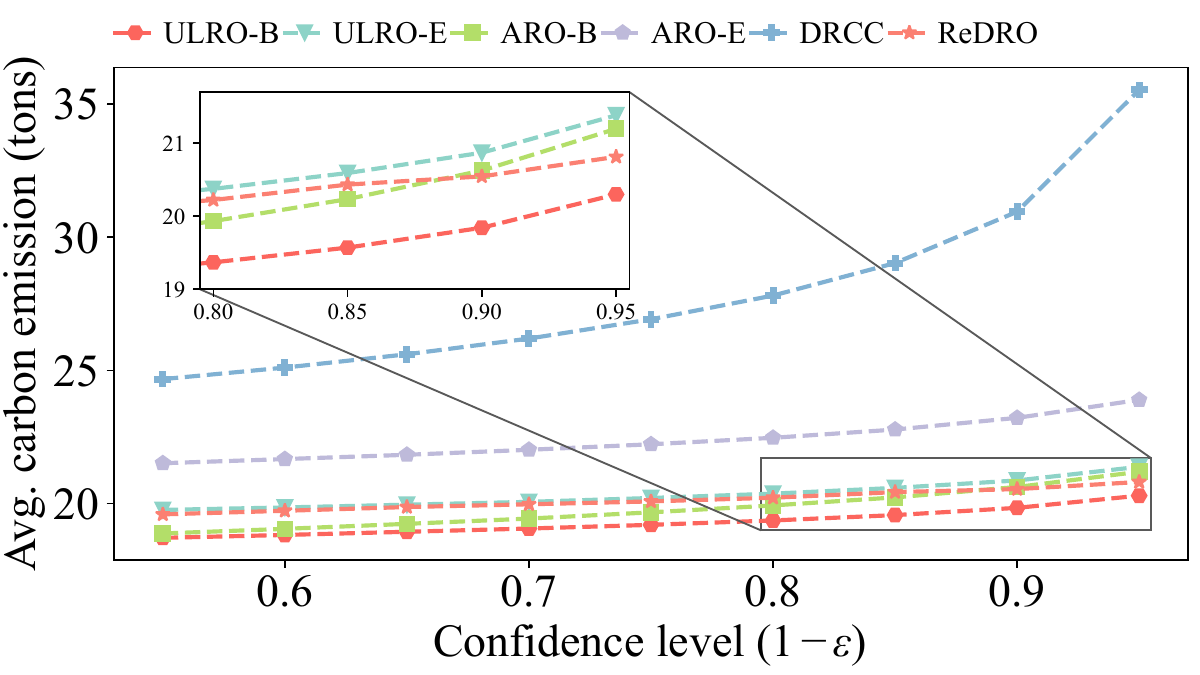}
\caption{Average carbon emissions of the data center with different methods and coverage levels (tons)}
\label{fig:carbon_epison}
\end{figure}

\begin{figure}[t]
\centering
\includegraphics[scale=0.35]{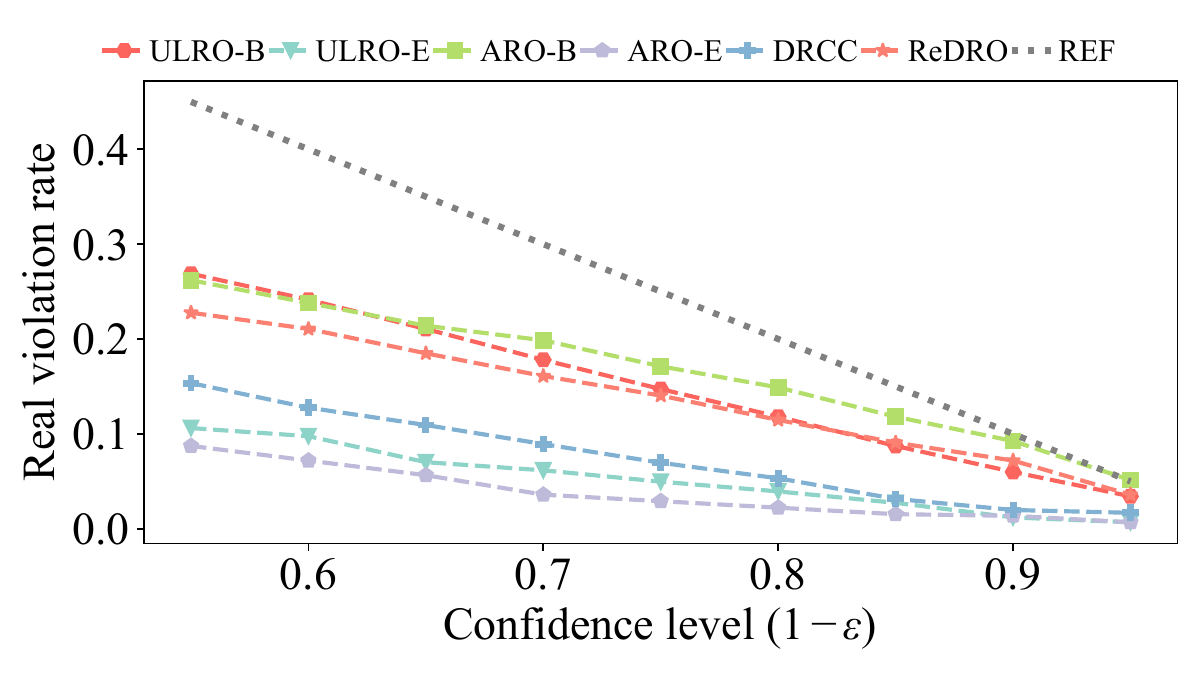}
\caption{Probability of constraint violation in the chance constraints under different methods}
\label{fig:vio_epison}
\end{figure}

\subsection{Experiments}
1)\textit{ Compared with benchmarks.} 
Figs.~\ref{fig:cost_epison} and \ref{fig:carbon_epison} present the average operating cost and carbon emissions under different confidence levels $1-\epsilon$, where the carbon price is set to $\$0.1$/kg. Across all confidence levels, the proposed ULRO approach yields the most favorable cost-emission performance among the compared methods. In particular, when $1-\epsilon=0.9$, compared with the corresponding covariate-independent method ARO-B, ULRO-B reduces the operational cost and carbon emissions by 6.67\% and 3.78\%, respectively. Compared with covariate-dependent ReDRO, ULRO-B reduces the operational cost and carbon emissions by 5.57\% and 3.41\%, respectively. The results demonstrate the advantage of incorporating covariate information into uncertainty-set construction. The proposed covariate-dependent methods, ULRO-B and ULRO-E, outperform their covariate-independent counterparts, ARO-B and ARO-E, respectively.  Under stringent reliability requirements, i.e., when $1-\epsilon$ exceeds 0.9, ULRO-B and ReDRO achieve better performance than other covariate-independent methods. 
Fig. \ref{fig:vio_epison} compares the empirical constraint violation probabilities obtained by different methods. The results show that ULRO achieves this feasibility without sacrificing economic performance, thereby providing a more effective balance between robustness and operational efficiency.  

\begin{table}[!t] \centering \caption{Running time under different sample sizes (s)} \label{tab:time_sample} \centering \resizebox{0.9\columnwidth}{!}{ \begin{tabular}{cccccc} \toprule $N$ & DRCC &ARO-B & ReDRO & CVaR-DJCC & \textbf{ULRO-B} \\ \midrule 100 & 4.26 & 0.77 & 5.91 & 48.97 & 0.73 \\ 200 & 4.18 & 0.70 & 10.96 & 155.37 & 0.81 \\ 300 & 4.22 & 0.77 & 16.13 & 373.75 & 0.79 \\ 400 & 4.32 & 0.76 & 21.12 & 696.93 & 0.78 \\ 500 & 4.22 & 0.77 & 26.14 & 1100.33 & 0.72 \\ 1000 & 4.31 & 0.69 & 52.54 & - & 0.78 \\ \bottomrule \end{tabular}} \end{table} 

\begin{table}[!t] \centering \caption{Running time under different Num. of DCs (s)} \label{tab:time_DC_nums} \centering \resizebox{0.9\columnwidth}{!}{ \begin{tabular}{cccccc} \toprule $|D|$ & DRCC &ARO-B & ReDRO & CVaR-DJCC & \textbf{ULRO-B}\\ \midrule 10 & 14.90 & 2.61 & 72.13 & 391.89 & 2.55 \\ 20 & 21.24 & 5.45 & 135.00 & 716.65 & 5.49 \\ 30 & 35.39 & 8.58 & 195.53 & 841.84 & 8.53 \\ 40 & 46.00 & 11.13 & 259.35 & 1092.56 & 11.71 \\ 50 & 52.59 & 13.43 & 324.91 & 1291.59 & 13.16 \\ 100 & 118.93 & 27.19 & 647.59 & - & 27.15 \\ \bottomrule 
\end{tabular}} \end{table}

2) \textit{Scalability.} 
Tables~\ref{tab:time_sample} and \ref{tab:time_DC_nums} show the average computational time of the compared methods under different sample sizes and numbers of data centers, respectively. ULRO-B and ARO-B achieve the lowest computational costs. Their solution times remain nearly stable as the sample size increases and grow only moderately with the number of data centers. In contrast, ReDRO and CVaR-DJCC require substantially longer solution times, particularly under large sample sizes and large-scale data center systems. 
Compared with other covariate-dependent methods, ULRO-B exhibits much better scalability. For example, with 100 data centers, ReDRO requires 647.59~s, whereas ULRO-B requires only 27.15~s, making ULRO-B approximately 24 times faster. 
This demonstrates that our ULRO is well-suited for large-scale data center scheduling with multiple workload classes and renewable energy sources.

\begin{figure}[t]
    \centering
    \label{fig:carbon_cost_aware}
    \begin{subfigure}[b]{0.4\textwidth} 
        \centering
    \includegraphics[width=\textwidth]{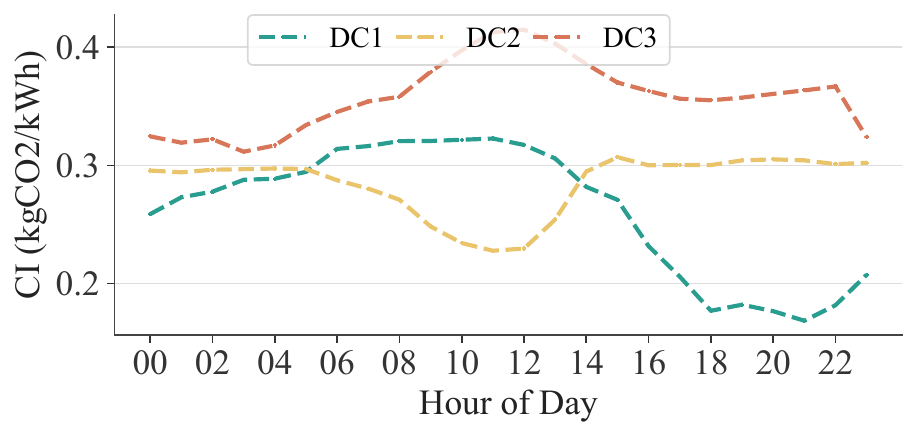} 
        \caption{Carbon intensity (kgCO2/kWh)}
    \label{fig:carbon_intensity}
    \end{subfigure}
    
    \vspace{-0.1em} 
    
    \begin{subfigure}[b]{0.4\textwidth}
        \centering
        \includegraphics[width=\textwidth]{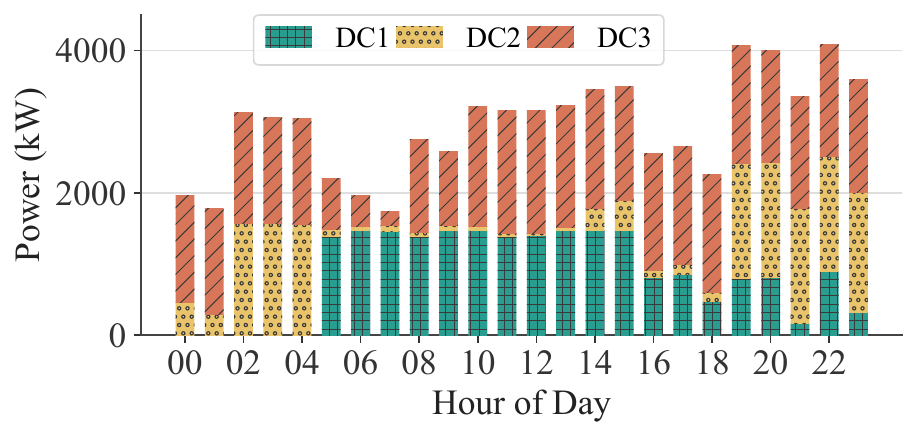}
        \caption{Cost-aware results}
\label{fig:cost_aware_workload}
    \end{subfigure}
        \vspace{-0.1em} 
      \begin{subfigure}[b]{0.4\textwidth}
        \centering
        \includegraphics[width=\textwidth]{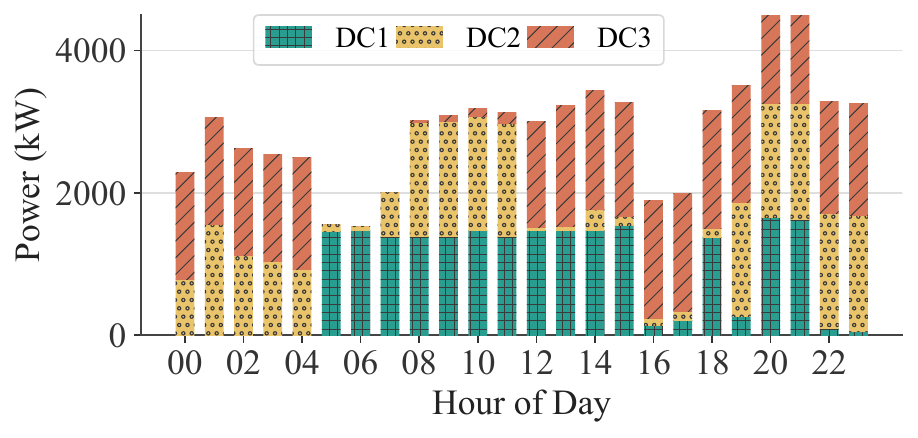}
        \caption{Carbon-aware results}
\label{fig:carbon_aware_workload}
    \end{subfigure} 
        \caption{Energy consumption profiles in cost-aware and carbon-aware scenarios.}
    \label{fig:carbon_cost_aware}
\end{figure}

3) \textit{Impacts of carbon price.} Tables~\ref{tab:carbon_price_obj} and \ref{tab:carbon_price_carbon} show the average operational cost and carbon emissions under different carbon prices, respectively. A higher carbon price increases the total cost but consistently reduces emissions, indicating that carbon pricing encourages cleaner scheduling decisions. Compared with the cost-aware case with zero carbon price, ULRO-B reduces carbon emissions by 2.12\% and 5.37\% under carbon prices of \$0.1/kg and \$0.2/kg, respectively. This trend is further illustrated in Fig.~\ref{fig:carbon_cost_aware}. DC1 becomes much cleaner during 18:00-24:00, whereas DC2 is cleaner during 8:00-14:00. The cost-aware schedule in Fig.~\ref{fig:cost_aware_workload} allocates workloads mainly according to electricity cost, which results in workload placement at high-carbon data centers and periods. By contrast, the carbon-aware schedule in Fig.~\ref{fig:carbon_aware_workload} shifts more workloads to cleaner locations and hours, such as allocating more demand to DC1 during its low-carbon evening period and reducing the use of higher-carbon DC3 in these intervals. 

\begin{table}[!t]
\centering
\caption{Average total cost by carbon price (k\$)}
\label{tab:carbon_price_obj}
\resizebox{0.9\columnwidth}{!}{
\begin{tabular}{cccccc}
\toprule
$p_{co_2}$ (\$/kg) & 
DRCC  &ARO-B & ReDRO &\textbf{ULRO-B}  \\
\midrule
0.0 & 24.38 & 15.18 & 15.01 & \textbf{14.10}\\
0.1 & 27.49 & 17.26 & 17.07 & \textbf{16.11}  \\
0.2 & 30.55 & 19.30 & 19.09 & \textbf{18.07} \\
\midrule
\end{tabular}}
\end{table}

\begin{table}[!t]
\centering
\caption{Average carbon emission by carbon price (tons)}
\label{tab:carbon_price_carbon}
\resizebox{0.9\columnwidth}{!}{
\begin{tabular}{ccccc}
\toprule
$p_{co_2}$ (\$/kg) & 
DRCC  &ARO-B & ReDRO & \textbf{ULRO-B}  \\
\midrule
0.0 & 31.23 & 20.96 & 20.67 & \textbf{20.27}  \\
0.1 & 30.96 & 20.63 & 20.55 & \textbf{19.84}  \\
0.2 & 30.00 & 19.94 & 19.64 & \textbf{19.18}  \\
\midrule
\end{tabular}}
\end{table}

\begin{figure}[!t]
    \centering
    \begin{subfigure}[b]{0.4\textwidth} 
        \centering
        \includegraphics[width=\textwidth]{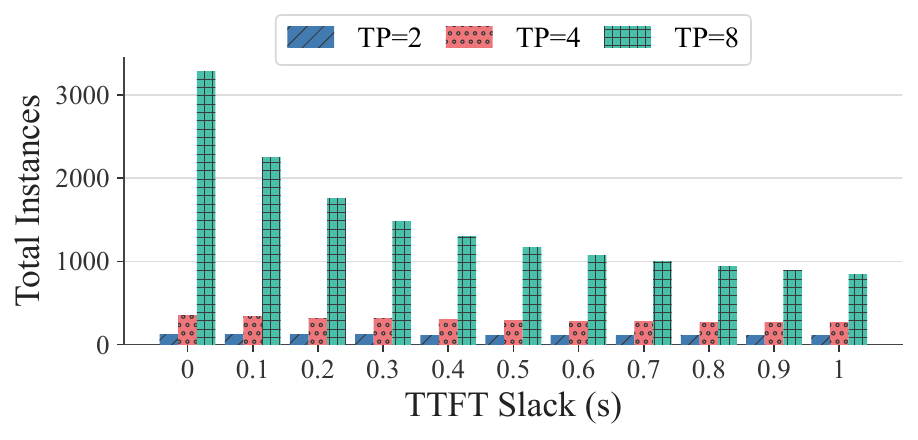} 
        \caption{Coding}
        \label{coding_ttft}
    \end{subfigure}
    
    \vspace{-0.1em} 
    
    \begin{subfigure}[b]{0.4\textwidth}
        \centering
        \includegraphics[width=\textwidth]{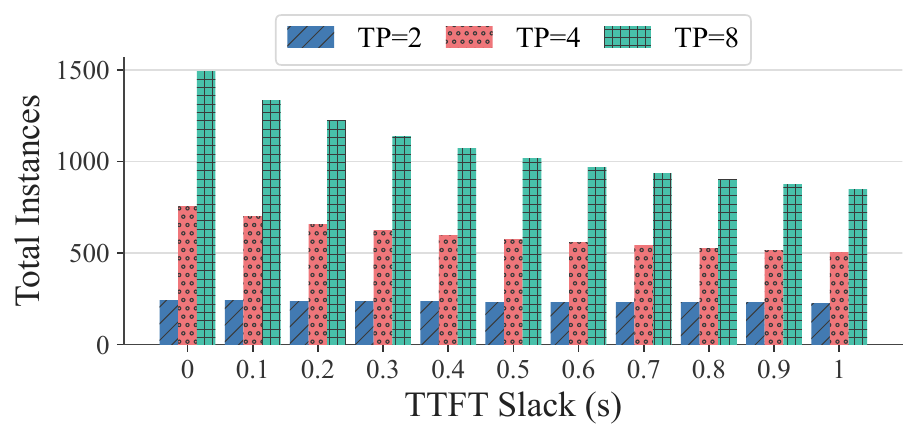}
        \caption{Conversation}
        \label{ttft_conversation}
    \end{subfigure}
    \caption{Number of instances of coding and conversation inference tasks under different TTFT slack during one day}

\end{figure}

4) \textit{Impacts of TTFT latency SLO of inference workloads.} We define TTFT slack as the additional time allowed before generating the first output token. A larger slack indicates a more relaxed latency requirement. Tables~\ref{tab:cdt_carbon} and \ref{tab:cdt_obj} show the power consumption of inference tasks and average carbon emissions under different TTFT slack values, respectively. As TTFT slack increases, both power consumption and carbon emissions of inference clusters decrease. 
This trend is explained in Figs.~\ref{coding_ttft} and \ref{ttft_conversation}. Under tight latency requirements, data centers must provision more inference instances and adopt higher TP-degree configurations, i.e., using more GPUs per request, to reduce first-token latency. These configurations improve serving speed but increase power consumption. Since coding tasks typically have longer input sequences than conversation tasks, they require more instances and higher TP degrees to satisfy the TTFT SLO.

\begin{table}[!t]
\centering
\caption{Average carbon emission of inference tasks by TTFT slack and method (tons)}
\label{tab:cdt_carbon}
\resizebox{0.9\columnwidth}{!}{
\begin{tabular}{ccccc}
\toprule
TTFT Slack (s) & 
DRCC  &ARO-B & ReDRO &ULRO-B  \\
\midrule
0.0 & 14.03 & 8.29 & 7.78 & \textbf{8.02} \\
0.5 & 12.13 & 7.04 & 6.70 & \textbf{6.79}  \\
1.0 & 11.75 & 6.79 & 6.47 & \textbf{6.53}  \\
\hline
\end{tabular}}
\end{table}

\begin{table}[!t]
\centering
\caption{Average cost of inference tasks by TTFT slack and method (k\$)}
\label{tab:cdt_obj}
\resizebox{0.9\columnwidth}{!}{
\begin{tabular}{ccccc}
\toprule
TTFT Slack (s) & 
DRCC  &ARO-B & ReDRO &\textbf{ULRO-B}  \\
\midrule
0.0 & 23.05 & 13.10 & 12.20 & \textbf{12.57} \\
0.5 & 19.42 & 11.02 & 10.33 & \textbf{10.57} \\
1.0 & 18.70 & 10.61 & 9.95  & \textbf{10.17}  \\
\hline
\end{tabular}}
\end{table}

\section{Conclusion}
\label{conclusion}
Decarbonizing AI data centers is essential for supporting sustainable computing. This paper proposes a contextual robust optimization framework for carbon-aware AI data center operation. The proposed model captures the heterogeneous computing characteristics of AI training and inference workloads and jointly optimizes delay-tolerant training workload allocation and latency-sensitive inference request routing. To ensure reliable operation under forecast errors in renewable generation and workload arrivals, the scheduling problem was first formulated as a contextual CCP with multiple JCCs.  We further transformed them into a tractable RO formulation using calibrated covariate-dependent uncertainty sets. Numerical studies based on real-world AI workload traces and renewable generation data demonstrate that the proposed method achieves reliable feasibility performance, improved operational efficiency, and strong scalability compared with benchmark methods.

\bibliographystyle{IEEEtran}
\bibliography{mybib.bib}

\section*{Appendix}
\setcounter{equation}{0}  
\renewcommand{\theequation}{A.\arabic{equation}}
\subsection{Proof for Theorem 1}

\begin{proof}
Let $\boldsymbol{x}$ be any feasible solution of the robust problem \eqref{eq:ro_multiple_jcc}. Then, by the robust feasibility condition, for every joint chance constraint $l$, we have
\begin{align}
g_{l,j}(\boldsymbol{x},\tilde{\boldsymbol{\xi}}_l)\le 0,
\quad
\forall \tilde{\boldsymbol{\xi}}_l\in\mathcal{U}_l(\boldsymbol{\psi}_l),
\quad
\forall j\in[m_l].
\end{align}
Therefore, whenever the realization $\tilde{\boldsymbol{\xi}}_l$ belongs to the uncertainty set $\mathcal{U}_l(\boldsymbol{\psi}_l)$, all individual constraints in the $l$-th joint constraint are satisfied simultaneously. 

Equivalently,
\begin{align}
\left\{
\tilde{\boldsymbol{\xi}}_l\in\mathcal{U}_l(\boldsymbol{\psi}_l)
\right\}
\subseteq
\left\{
g_{l,j}(\boldsymbol{x},\tilde{\boldsymbol{\xi}}_l)\le 0,
\ \forall j\in[m_l]
\right\}.
\end{align}
Taking the conditional probability on both sides yields
\begin{align}
&\resizebox{0.88\linewidth}{!}{$
\mathbb{P}_{\boldsymbol{\psi}_l}
\left(
g_{l,j}(\boldsymbol{x},\tilde{\boldsymbol{\xi}}_l)\le 0,
\ \forall j\in[m_l]
\right) \ge
\mathbb{P}_{\boldsymbol{\psi}_l}
\left(
\tilde{\boldsymbol{\xi}}_l\in\mathcal{U}_l(\boldsymbol{\psi}_l)
\right)
\ge 1-\epsilon_l$.}
\end{align}
Since the above argument holds for every $l$, $\boldsymbol{x}$ satisfies all contextual joint chance constraints in \eqref{eq:standard_ccp}. Therefore, any feasible solution of \eqref{eq:ro_multiple_jcc} is feasible for the contextual chance-constrained problem \eqref{eq:standard_ccp}.
\end{proof}

\setcounter{equation}{0}  
\renewcommand{\theequation}{B.\arabic{equation}}
\subsection{Proof for Theorem 2}
\begin{proof}
Let $s_k^n$ denote the calibration score of the sample 
$(\boldsymbol{X}^n,\boldsymbol{y}^n)$, and let 
$s_k^{\mathrm{N+1}}$ denote the score of the new sample. Specifically, $s_k^n$ is given by the joint interval violation score for the box set and by the normalized residual score for the ellipsoidal set.

By assumption, the calibration samples and the new sample are exchangeable. Since each score is obtained by applying the same deterministic scoring function to each sample, the scores
\[
s_k^1,\dots,s_k^{N},s_k^{N+1}
\]
are also exchangeable. Let
\[
r_{\epsilon}
=
\left\lceil
(N+1)(1-\epsilon)
\right\rceil ,
\]
and define the calibrated threshold as the $r_{\epsilon}$-th order statistic of the calibration scores,
\[
q_k = s_{k,(r_{\epsilon})}.
\]
By the exchangeability of the 
$N+1$ scores, the rank of $s_k^{N+1}$ among 
$\{s_k^1,\dots,s_k^{N},s_k^{N+1}\}$ is uniformly distributed over 
$\{1,\dots,N+1\}$, up to ties. Therefore,
\begin{align}
\mathbb{P}
\left(
s_k^{N+1} \leq q_k
\right)
\geq
\frac{r_{\epsilon}}{N+1}
\geq
1-\epsilon .
\end{align}

It remains to connect the score event to set membership. By the construction of the calibrated uncertainty sets, the event 
$s_k^{N+1}\leq q_k$ is equivalent to 
\[
\boldsymbol{y}^{N+1}
\in
\mathcal{U}_{\epsilon}^{k}
(\boldsymbol{X}^{N+1}).
\]
Hence,
\begin{align}
\mathbb{P}
\left(
\boldsymbol{y}^{n+1}
\in
\mathcal{U}_{\epsilon}^{k}
(\boldsymbol{X}^{n+1})
\right)
=
\mathbb{P}
\left(
s_k^{N+1}\leq q_k
\right)
\geq
1-\epsilon .
\end{align}
\end{proof}

\end{document}